\documentclass[reqno,english]{amsart}
\usepackage{amsfonts,amsmath,latexsym,verbatim,amscd,mathrsfs,color,array}
\usepackage[colorlinks=true]{hyperref}

\usepackage{amsmath,amssymb,amsthm,amsfonts,graphicx,color}
\usepackage{amssymb}
\usepackage{pdfsync}
\usepackage{epstopdf}

\usepackage{cite}

 \usepackage{graphicx}

\newtheorem{theorem}{Theorem}[section]

\newtheorem{lemma}[theorem]{Lemma}

\newtheorem{proposition}[theorem]{Proposition}
\newtheorem{remark}[theorem]{Remark}

\numberwithin{equation}{section}

\begin{document}

\title[Gross-Pitaevskii and Adler-Moser Polynomials]{Multi-vortex traveling waves for the Gross-Pitaevskii equation and the
Adler-Moser polynomials}

\author[Y.Liu]{Yong Liu}

\address{\noindent School of Mathematics and Physics, North China Electric Power University, Beijing, China}
\email{liuyong@ncepu.edu.cn}

\author[J. Wei]{Juncheng Wei}
\address{\noindent
Department of Mathematics,
University of British Columbia, Vancouver, B.C., Canada, V6T 1Z2}
\email{jcwei@math.ubc.ca}

\maketitle

\begin{abstract}
For $N\leq34,$ we construct traveling waves with small speed for the
Gross-Pitaevskii equation, by gluing $N(N+1)/2$ pairs of degree $\pm1$
vortices of the Ginzburg-Landau equation. The location of these vortices is
symmetric in the plane and determined by the Adler-Moser polynomials, which
has its origin in the study of Calogero-Moser system and rational solutions of
the KdV equation. The construction still works for $N>34$, under the additional
assumption that the corresponding Adler-Moser polynomial has no repeated root.
It is expected that this assumption holds for any $N\in\mathbb{N}$.

\end{abstract}

\setcounter{equation}{0}
\section{Introduction}

The Gross-Pitaevskii equation (GP equation) arises as a model equation in
Bose-Einstein condensate. It reads as%
\begin{equation}
i\partial_{t}\Phi=\Delta\Phi+\Phi\left(  1-\left\vert \Phi\right\vert
^{2}\right)  ,\text{ in }\mathbb{R}^{2}, \label{GP}%
\end{equation}
where $\Phi$ is complex valued. Throughout the paper, $i$ will represent the
imaginary unit. For traveling wave solutions of the form $U\left(
x,y-\varepsilon t\right),$ GP equation becomes
\begin{equation}
-i\varepsilon\partial_{y}U=\Delta U+U\left(  1-\left\vert U\right\vert
^{2}\right) ,\text{ in }\mathbb{R}^{2}. \label{TW}%
\end{equation}

\medskip

We would like to construct multi-vortex type solutions of $\left(
\ref{TW}\right)  $ when the speed $\varepsilon$ is close to zero. If
$\varepsilon=0,$ then the above equation reduces to the well-known
Ginzburg-Landau equation%
\begin{equation}
\Delta U+U\left(  1-\left\vert U\right\vert ^{2}\right)  =0 ,\text{ in }\mathbb{R}^{2}.
\end{equation}

\medskip

We shall use $\left(r,\theta\right)  $ to denote the polar coordinate of
$\mathbb{R}^{2}$. For each $d\in\mathbb{Z}\backslash\left\{  0\right\}  ,$ it
is known that the Ginzburg-Landau equation has a degree $d$ vortex solution,
of the form $e^{id\theta}S_{d}\left(  r\right)  $. The function $S_{d}$ is
real valued and vanishes exactly at $r=0.$ It satisfies%
\[
-S_{d}^{\prime\prime}-\frac{1}{r}S_{d}^{\prime}+\frac{d^{2}}{r^{2}}S_{d}%
=S_{d}\left(  1-S_{d}^{2}\right)  ,\text{ in }\left(  0,+\infty\right)  .
\]
This equation has a unique solution $S_{d}$ satisfying $S_{d}\left(  0\right)
=0$ and $S_{d}\left(  +\infty\right)  =1$ and $S^{\prime}\left(  r\right)
>0.$ See \cite{Fife,Her} for a proof. The \textquotedblleft
standard\textquotedblright\ degree $\pm1$ solutions $S_{1}\left(  r\right)
e^{i\theta}$ are global minimizers of the energy functional(For uniqueness of
the global minimizer, see \cite{M2,SH}). When $\left\vert d\right\vert >1,$
these standard vortices are unstable(\cite{Mironescu,Lin}). It is also worth
mentioning that for $\left\vert d\right\vert >1,$ the uniqueness of degree $d$
vortex in the class of solutions with degree $d$ is still an open problem. We
refer to \cite{Brezis,Pacard,Serfaty} and the references therein for more
discussion on the Ginzburg-Landau equation.

\medskip

The constant $1$ is a solution to the equation $\left(  \ref{TW}\right)  $. We
are interested in the solution $U$ which satisfies%
\[
U (z) \rightarrow1\text{, as }\left\vert z\right\vert \rightarrow+\infty.
\]
The existence or nonexistence of solutions to (\ref{TW}) with this asymptotic behavior has
been extensively studied in the literature. Jones, Putterman,
Roberts(\cite{J,J1}) studied it from the physical point of view, both in
dimension two and three. It turns out the existence of solutions is related to
the traveling speed $\varepsilon.$ When $\varepsilon\geq\sqrt{2}$ (the sound
speed in this context), nonexistence of traveling wave with \textit{finite
energy} is proved by Gravejat in \cite{G2,G3}. On the other hand, for
$\varepsilon\in\left(  0,\sqrt{2}\right)  ,$ the existence of travelling waves
as constrained minimizer is studied by Bethuel, Gravejat, Saut \cite{Ber1,Ber3}%
, by variational arguments. For $\varepsilon$ close to $0,$ these solutions
have two vortices. The existence issue in higher dimension is studied
\cite{Ber2,Ch1,Ch2}. We also refer to \cite{B1} for a review on this subject.
Recently, Chiron-Scheid \cite{Chiron} performed numerical simulation on this
equation. Among other things, their results indicate the existence of higher
energy traveling waves. \ We also mention that as $\varepsilon$ tends to
$\sqrt{2},$ a suitable rescaled traveling waves will converge to solutions of
the KP-I equation (\cite{KP}), which is classical integrable system. In a
forthcoming paper, we will construct transonic traveling waves based on the
lump solution of the KP-I equation.

\medskip

Another motivation for studying (\ref{TW}) arises in the study of superfuilds passing an obstacle. Equation (\ref{TW}) is the limiting equation in the search of vortex nucleation solution. We refer to recent paper \cite{LW} for references and derivations.

\medskip

To simplify notations, we write the degree $\pm1$ vortex solutions of the
Ginzburg-Landau equation as
\[
v_{+}=e^{i\theta}S_{1}\left(  r\right)  ,v_{-}=e^{-i\theta}S_{1}\left(
r\right)  .
\]
In this paper, we construct new traveling waves for $\varepsilon$ close to
$0,$ using $v_{+},v_{-}$ as basic blocks. Our main result is

\begin{theorem}
\label{main}For each $N\leq34,$ there exists $\varepsilon_{0}>0,$ such that
for all $\varepsilon\in\left(  0,\varepsilon_{0}\right)  ,$ the equation
$\left(  \ref{TW}\right)  $ has a solution $U_{\varepsilon}$ which has the
form
\[
U_{\varepsilon}=\prod\limits_{k=1}^{N\left(  N+1\right)  /2}\left(
v_{+}\left(  z-\varepsilon^{-1}p_{k}\right)  v_{-}\left(  z-\varepsilon
^{-1}q_{k}\right)  \right)  +o\left(  1\right)  ,
\]
where $p_{k}$, $k=1,...,N\left(  N+1\right)  /2$ are the roots of the
Adler-Moser polynomial $A_{n}$ defined in the next section ($n=\frac{N(N+1)}{2}$), and $q_{k}%
=-\bar{p}_{k}.$
\end{theorem}

\begin{remark}
$N=1$ corresponds to the two vortices solutions constructed by variational method (\cite{Ber3}) as well as reduction method (\cite{LFH}). For large $N$ this corresponds to higher energy solutions which have been observed numerically \cite{Chiron}.
\end{remark}

\begin{remark}
The condition $N\leq34$ is only technical. In this case, we can verify
numerically that the Adler-Moser polynomial $A_{n}$ has no repeated roots. (This
is equivariant to $A_{n}$ and $A_{n-1}$ have no common roots.) Possibly, there
are other numerically ways to verify this for large $N$(using the recursive
identity to compute the Adler-Moser polynomial, instead of computing the
Wronskian), but we will not pursue this here. We conjecture that the Adler-Moser polynomial (as constructed in this paper) has only simple roots for all $N$.
\end{remark}

\medskip

Our method is based on finite dimensional Lyapunov-Schmidt reduction. We show
that the existence of multi-vortex solutions is essentially reduced to the
study of the non-degeneracy of a symmetric vortex-configuration. To show this
non-degeneracy, we use the theory of Adler-Moser polymonials and the Darboux
transformation. An interesting feature of the solutions in Theorem \ref{main}
is that the vortex location has a ring-like structure. The emergence of this
remarkable property still remains mysterious.

\medskip

In Section \ref{AM}, we introduce the Adler-Moser polynomials and prove the
non-degeneracy of the symmetric configuration. In Section \ref{Pri}, we recall the linear
theory of the degree one vortex of the Ginzburg-Landau equation. In Section
\ref{non}, we use Lyapunov-Schmidt reduction to glue the vortices together and
get a traveling wave solution for $\varepsilon$ small enough.

\medskip

\textbf{Acknowledgement } The research of J. Wei is partially supported by
NSERC of Canada. Part of this work is finished while the first author is
visiting the University of British Columbia in 2017. He thanks the institute
for the financial support. Both authors thank Professor Fanghua Lin for stimulating discussions and suggestions.

\setcounter{equation}{0}
\section{\label{AM}Vortex location and the Adler-Moser polynomials}

Adler-Moser\cite{Adler} have studied a set of polynomials corresponding to
rational solutions of the KdV equation. It turns out that these polynomials
have deep connections to the vortex dynamics with logrithmic interaction
energy. This connection is first observed in \cite{Bartman}, and later studied
in \cite{Aref,Aref2,Aref3,Cla,Ka}. (See the reference therein.)
Surprisingly, to authors' knowledge, except the papers mentioned above, it
seems that this relation has not been much explored  for the corresponding PDEs.  One of our aims in this paper is to fill this gap.

\medskip

In this section,
we will first recall some basic facts about these polynomials and then analyze
some of their properties, which will be used in our construction of the
traveling wave for the GP equation.

\medskip

Let $p_{1},...,p_{n}$ be the
position of the positive vortices and $q_{1},...,q_{m}$ be that of the
negative ones. Let $\mu\in\mathbb{R}$ be a fixed parameter. As we will see
later, the vortex location of the traveling waves will be determined by the
following systems of equations
\begin{equation}
\left\{
\begin{array}
[c]{c}%
\sum\limits_{j\neq\alpha}\frac{1}{p_{\alpha}-p_{j}}-\sum\limits_{j}\frac
{1}{p_{\alpha}-q_{j}}=\mu,\text{ for }\alpha=1,...,n,\\
\sum\limits_{j\neq\alpha}\frac{1}{q_{\alpha}-q_{j}}-\sum\limits_{j}\frac
{1}{q_{\alpha}-p_{j}}=-\mu,\text{ for }\alpha=1,...,m.
\end{array}
\right.  \label{trans}%
\end{equation}
Adding all these equation together, we find that if $\mu\neq0,$ then $m=n$. (In
the case of $\mu=0,$ this is no longer true). That is, the number of positive
vortices must equal to that of the negative vortices. Solutions of this
system (see for instances \cite{Aref3}) are related to the Adler-Moser
polynomials. To explain this, let us define the generating polynomials%
\[
P\left(  z\right)  =\prod\limits_{j}\left(  z-p_{j}\right)  ,\text{ }Q\left(
z\right)  =\prod\limits_{j}\left(  z-q_{j}\right)  .
\]
If $p_{j},q_{j}$ satisfy $\left(  \ref{trans}\right)  ,$ then we have (see
equation (68) of \cite{Aref3})
\begin{equation}
P^{\prime\prime}Q-2P^{\prime}Q^{\prime}+PQ^{\prime\prime}=2\mu\left(
P^{\prime}Q-PQ^{\prime}\right)  . \label{pq}%
\end{equation}
This equation is usually called generalized Tkachenko equation. Setting
$\psi\left(  z\right)  =\frac{P}{Q}e^{\mu z},$ we derive from $\left(
\ref{pq}\right)  $ that
\[
\psi^{\prime\prime}+2\left(  \ln Q\right)  ^{\prime\prime}\psi=\mu^{2}\psi.
\]

\medskip

This is a one dimensional Schrodinger equation with the potential $2\left(
\ln Q\right)  ^{\prime\prime}.$ It is well known that this equation appears in
the Lax pair of the KdV equation. Hence equation $\left(  \ref{pq}\right)  $
is naturally related to the theory of integrable systems.

\medskip

For any $z\in\mathbb{C},$ we use $\bar{z}$ to denote its complex conjugate. To
simplify the notation, we also write $-\bar{z}$ as $z^{\ast}.$ Note that this
is just the reflection of $z$ across the $y$ axis. Let $K=\left(
k_{2},...\right)  ,$ where $k_{i}$ are complex parameters. Following
\cite{Cla}, we define functions $\theta_{n},$ depending on $K,$ by%
\begin{equation}
\sum_{n=0}^{+\infty}\theta_{n}\left(  z;K\right)  \lambda^{n}=\exp\left(
z\lambda-\sum_{j=2}^{\infty}\frac{k_{j}\lambda^{2j-1}}{2j-1}\right)  .
\end{equation}
Note that $\theta_{n}$ is a degree $n$ polynomial in $z$ and $\theta
_{n+1}^{\prime}=\theta_{n}.$ Explicitly, $\theta_{1}\left(  z;K\right)  =z,$%
\[
\theta_{3}\left(  z;K\right)  =\frac{z^{3}}{6}-\frac{k_{2}}{3},
\]%
\[
\theta_{5}\left(  z;K\right)  =-\frac{k_{3}}{5}-\frac{k_{2}}{6}z^{2}+\frac
{1}{120}z^{5}.
\]
Let $c_{n}=\prod\limits_{j=1}^{n}\left(  2j+1\right)  ^{n-j}.$ For each
$n\in\mathbb{N},$ the Adler-Moser polynomials are then defined by%
\begin{equation}
\Theta_{n}\left(  z,K\right)  :=c_{n}W\left(  \theta_{1},\theta_{3}%
,...,\theta_{2n-1}\right)  ,
\end{equation}
where $W\left(  \theta_{1},\theta_{3},...,\theta_{2n-1}\right)  $ is the
Wronskian of $\theta_{1},...,\theta_{2n-1}.$ In particular, the degree of
$\Theta_{n}$ is $n\left(  n+1\right)  /2.$ The constant $c_{n}$ is chosen
such that the leading coefficient of $\Theta_{n}$ is $1.$ The first three
Adler-Moser polynomials are $\Theta_{1}\left(  z,K\right)  =z,$ $\Theta
_{2}\left(  z,K\right)  =z^{3}+k_{2},$ and
\[
\Theta_{3}\left(  z,K\right)  =z^{6}+5k_{2}z^{3}-9k_{3}z-5k_{2}^{2}.
\]
Note that this definition is slightly different from that of
Adler-Moser\cite{Adler}. (The parameter $\tau_{i}$ in that paper is different
from $k_{i}$ here.)

\medskip

Let $\mu$ be another parameter. The modified Adler-Moser polynomial
$\tilde{\Theta}$ is defined by
\begin{equation}
\tilde{\Theta}_{n}\left(  z,\mu,K\right)  :=c_{n}e^{-\mu z}W\left(  \theta
_{1},\theta_{3},...,\theta_{2n-1},e^{\mu z}\right)
\end{equation}
where $\tilde{K}=\left(  k_{2}+\mu^{-3},k_{3}+\mu^{-5},...,k_{n}+\mu
^{-2n+1}\right)  .$ It is still a polynomial in $z$ with degree $n\left(
n+1\right)  /2.$ We observe that for a given $\mu,$ $\Theta_{n}$ depends on
$n-1$ complex parameters $k_{2},...,k_{n}.$ This together with the translation
in $z$ give us a total of $n$ complex parameters.

The following result, stated without proof in \cite{Cla}, will play an
important role in our later analysis.

\begin{lemma}
The Adler-Moser and modified Adler-Moser polynimials are related by
\[
\tilde{\Theta}_{n}\left(  z,\mu,K\right)  =\mu^{n}\Theta_{n}\left(  z-\mu
^{-1},\tilde{K}\right)  .
\]

\end{lemma}

\begin{proof}
We sketch the proof for completeness. First of all, direction computation
shows that%
\[
\sum_{n=0}^{+\infty}\theta_{n}\left(  z;K\right)  \lambda^{n}=\sqrt
{\frac{1+\mu^{-1}\lambda}{1-\mu^{-1}\lambda}}\sum_{n=0}^{+\infty}\theta
_{n}\left(  z-\mu^{-1};\tilde{K}\right)  \lambda^{n}.
\]
From this we obtain%
\[
\mu^{-1}\sum_{n=0}^{+\infty}\theta_{n-1}\left(  z;K\right)  \lambda^{n}%
=\mu^{-1}\lambda\sqrt{\frac{1+\mu^{-1}\lambda}{1-\mu^{-1}\lambda}}\sum
_{n=0}^{+\infty}\theta_{n}\left(  z-\mu^{-1};\tilde{K}\right)  \lambda^{n}.
\]
Hence using the fact that $\theta_{n}^{\prime}=\theta_{n-1},$ we get
\begin{align*}
&  \sum_{n=0}^{+\infty}\left(  \theta_{n}\left(  z;K\right)  -\mu^{-1}%
\theta_{n}^{\prime}\left(  z;K\right)  -\theta_{n}\left(  z-\mu^{-1};\tilde
{K}\right)  \right)  \lambda^{n}\\
&  =\left(  \sqrt{\frac{1+\mu^{-1}\lambda}{1-\mu^{-1}\lambda}}-1-\mu
^{-1}\lambda\sqrt{\frac{1+\mu^{-1}\lambda}{1-\mu^{-1}\lambda}}\right)
\sum_{n=0}^{+\infty}\theta_{n}\left(  z-\mu^{-1};\tilde{K}\right)  \lambda
^{n}.
\end{align*}
We observe that
\[
\sqrt{\frac{1+\mu^{-1}\lambda}{1-\mu^{-1}\lambda}}-1-\mu^{-1}\lambda
\sqrt{\frac{1+\mu^{-1}\lambda}{1-\mu^{-1}\lambda}}=\sqrt{1-\mu^{-2}\lambda
^{2}}-1.
\]
The Taylor expansion of this function contains only even powers of $\lambda.$
Hence for odd $n,$ $\theta_{n}\left(  z;K\right)  -\mu^{-1}\theta_{n}^{\prime
}\left(  z;K\right)  -\theta_{n}\left(  z-\mu^{-1};\tilde{K}\right)  $ can be
written as a linear combination of $\theta_{k}\left(  z-\mu^{-1};\tilde
{K}\right)  $ with $k$ being odd. The desired identity then follows.
\end{proof}

The next result, which essentially follows from Crum type theorem, reveals the
relation of the Adler-Moser polynomial with the vortex dynamics (\cite{Aref3},
see also Theorem 3.3 in \cite{Cla}).

\begin{lemma}
The functions $Q=\Theta_{n}\left(  z,K\right)  ,P=\tilde{\Theta}_{n}\left(
z,\mu,K\right)  $ satisfy $\left(  \ref{pq}\right)  .$
\end{lemma}

Note that a general degree $m$ term in $\theta_{n}$ has the form $k_{2}%
^{l_{2}}\cdot\cdot\cdot k_{j}^{l_{j}}z^{m}.$ We define its index to be
$\left(  -1\right)  ^{l_{2}+...+l_{j}+m}.$ We now prove the following

\begin{lemma}
\label{index}\bigskip For each term of $\theta_{2n+1},$ its index is $-1.$
\end{lemma}

\begin{proof}
Let $k_{2}^{l_{2}}\cdot\cdot\cdot k_{j}^{l_{j}}z^{m}$ be a degree $m$ term in
$\theta_{2n+1}.$ By Taylor expansion of the generating function and using the
fact that $2n+1$ is odd, this term comes from functions of the form,
\[
\frac{1}{\alpha!}\left(  z\lambda-\sum_{j=2}^{\infty}\frac{k_{j}\lambda
^{2j-1}}{2j-1}\right)  ^{\alpha},
\]
where $\alpha$ is an odd integer. Hence the total degree of $k_{j}$ is
$\alpha-m.$ Then the index is $\left(  -1\right)  ^{\alpha-2m}=-1.$
\end{proof}

\begin{lemma}
\label{Theta}For each term of $\Theta_{n},$ its index is equal to $\left(
-1\right)  ^{\frac{n\left(  n+1\right)  }{2}}.$
\end{lemma}

\begin{proof}
Let us consider a typical term of $\Theta_{n}$, say $\theta_{1}\theta
_{3}^{\prime}...\theta_{2n-1}^{\left(  n-1\right)  },$ where the notation
$\left(  {}\right)  $ represents taking derivatives. By Lemma \ref{index},
terms in $\theta_{k}^{\left(  j\right)  }$ have index $\left(  -1\right)
^{1+j}.$ Hence the index of terms in $\theta_{1}\theta_{3}^{\prime}%
...\theta_{2n-1}^{\left(  n-1\right)  }$ is $\left(  -1\right)  ^{1+2+...+n}%
=\left(  -1\right)  ^{\frac{n\left(  n+1\right)  }{2}}.$ This finishes the proof.
\end{proof}

Now we introduce the notation
\[
\Theta_{n,t}\left(  z,K\right)  :=\Theta_{n}\left(  z-t,K\right)  .
\]
For any polynomial $\phi$(in $z$), we use $R\left(  \phi\right)  $ to denote
the set of roots of $\phi.$ We have the following

\begin{lemma}
\label{symmetry}Suppose $\mu$ is a real number. Assume $t=-\frac{\mu}{2}$ and
$k_{j}=-\frac{1}{2}\mu^{2j-1}$ for $j=2,....$ Then
\begin{equation}
\left(  \Theta_{n,t}\left(  z,K\right)  \right)  ^{\ast}=\left(  -1\right)
^{\frac{n\left(  n+1\right)  }{2}+1}\tilde{\Theta}_{n,t}\left(  z^{\ast}%
,\mu^{-1},K\right)  .
\end{equation}
As a consequence, in this case, the reflection of $R\left(  \Theta
_{n,t}\left(  z,K\right)  \right)  $ across the $y$ axis is $R\left(
\tilde{\Theta}_{n,t}\left(  z,\mu,K\right)  \right)  ,$ and $R\left(
\Theta_{n,t}\left(  z,K\right)  \right)  $ is invariant respect to the
reflection across the $x$ axis.
\end{lemma}

\begin{proof}
By Lemma \ref{Theta}, for each term $f=k_{1}^{i_{1}}\cdot\cdot\cdot
k_{j}^{i_{j}}\left(  z-t\right)  ^{m}$ of the function $\Theta_{n,t}\left(
z,K\right)  $, we have in $\tilde{\Theta}_{n,t}\left(  z^{\ast},\mu,K\right)
$ a corresponding term $\tilde{k}_{1}^{i_{1}}\cdot\cdot\cdot\tilde{k}%
_{j}^{i_{j}}\left(  z^{\ast}-t-\mu\right)  ^{m},$ denoted by $g.$ By the
choice of $k_{j},$ we know that
\[
\tilde{k}_{j}=-k_{j}.
\]
By Lemma \ref{Theta}, the index of $k_{1}^{i_{1}}\cdot\cdot\cdot k_{j}^{i_{j}%
}z^{m}$ is $\left(  -1\right)  ^{\frac{n\left(  n+1\right)  }{2}}.$ Hence
using the fact that $\mu$ is real, we get
\begin{align*}
f^{\ast}  &  =-k_{1}^{i_{1}}\cdot\cdot\cdot k_{j}^{i_{j}}\left(  -z^{\ast
}-t\right)  ^{m}\\
&  =\left(  -1\right)  ^{1+i_{1}+...+i_{j}+m}\tilde{k}_{1}^{i_{1}}\cdot
\cdot\cdot\tilde{k}_{j}^{i_{j}}\left(  z^{\ast}+t\right)  ^{m}\\
&  =\left(  -1\right)  ^{\frac{n\left(  n+1\right)  }{2}+1}g.
\end{align*}
This completes the proof.
\end{proof}

Taking for example $\mu=2,t=-1,$ $k_{2}=-4,k_{3}=-16,$ we get $\Theta_{2,t}\left(
z,K\right)  =\left(  z+1\right)  ^{3}-4.$ $\allowbreak$It has one real root
and a pair of conjugate roots, forming a regular triangle, and given
numerically by
\begin{equation}
p_{1}=0.587\,4,p_{2}=-1.\,\allowbreak793\,7-1.\,\allowbreak374\,7i,p_{3}%
=-1.\,\allowbreak793\,7+1.\,\allowbreak374\,7i. \label{p}%
\end{equation}

In the sequel, for simplicity, we shall choose $\mu=1.$ Then $t=k_{j}%
=-\frac{1}{2}.$ Let us then denote the corresponding polynomial $\Theta
_{n,t}\left(  z,K\right)  $ by $A_{n}\left(  z\right)  .$ Then $\tilde{\Theta
}_{n}\left(  z,\mu,K\right)  $ is equal to $A_{n}\left(  -z\right)  ,$ which
we denote by $B_{n}\left(  z\right)  .$

\medskip

Since our traveling wave solutions will roughly speaking have vortices at the
roots of $A_{n},$ it is natural to ask that whether all the roots of $A_{n}$
are simple. This question seems to be nontrivial.

\begin{lemma}
\label{a}Let $P\left(  z\right)  ,Q\left(  z\right)  $ be two polynomials
satisfying
\begin{equation}
P^{\prime\prime}Q-2P^{\prime}Q^{\prime}+PQ^{\prime\prime}=2\mu\left(
P^{\prime}Q-PQ^{\prime}\right)  , \label{p1}%
\end{equation}
or
\begin{equation}
P^{\prime\prime}Q-2P^{\prime}Q^{\prime}+PQ^{\prime\prime}=0. \label{p2}%
\end{equation}
Suppose $P\left(  \xi\right)  =0,Q\left(  \xi\right)  \neq0,$ for some $\xi.$
Then $\xi$ is a simple root of $P.$
\end{lemma}

\begin{proof}
We prove the lemma assuming $\left(  \ref{p1}\right)  .$ The case of $\left(
\ref{p2}\right)  $ is similar.

Suppose $\xi$ is root of $P$ with multiplicity $k\geq2.$ We have
\[
P^{\prime\prime}Q=2P^{\prime}Q^{\prime}-PQ^{\prime\prime}+2\mu\left(
P^{\prime}Q-PQ^{\prime}\right)  .
\]
Then $\xi$ is a root of the right hand side polynomial with multiplicity at
least $k-1.$ But its multiplicity in $P^{\prime\prime}Q$ is $k-2.$ This is a contradiction.
\end{proof}

\begin{lemma}
\label{b}Suppose $P\left(  z\right)  ,Q\left(  z\right)  $ are two polynomials
satisfying $\left(  \ref{p1}\right)  $ or $\left(  \ref{p2}\right)  .$ Let
$\xi$ be a common root of $P$ and $Q.$ Assume $\xi$ is a simple root of $Q.$
Then $\xi$ can not be a simple root of $P.$
\end{lemma}

\begin{proof}
We prove this lemma assuming $\left(  \ref{p2}\right)  .$ The case of $\left(
\ref{p1}\right)  $ is similar.

Assume to the contrary that $\xi$ is a simple root of $P.$ Then%
\[
2P^{\prime}\left(  \xi\right)  Q^{\prime}\left(  \xi\right)  \neq0.
\]
But this contradicts with the equation $\left(  \ref{p2}\right)  .$ This
finishes the proof.
\end{proof}

We introduce the following assumption:

\noindent
(\textbf{A}). \textit{The polynomials }$A_{n}\left(  z\right)  $\textit{ and
}$A_{n-1}\left(  z\right)  $\textit{ have no common roots.}

\begin{lemma}
Suppose the assumption (A) holds. Then $A_{n}$ has no repeated roots.
Moreover, $A_{n}\left(  \xi\right)  $ and $A_{n}\left(  -\xi\right)  $ have no
common roots.
\end{lemma}

\begin{proof}
We know that the sequence of Adler-Moser polynomials satisfy the following
recursion relation%
\begin{equation}
A_{n}^{\prime\prime}A_{n-1}-2A_{n}^{\prime}A_{n-1}^{\prime}+A_{n}%
A_{n-1}^{\prime\prime}=0,\text{ for any }n. \label{anan}%
\end{equation}
By Lemma \ref{a}, any root of $A_{n}$ is a simple root. Similarly, any root of
$A_{n}\left(  -z\right)  $ is a simple root.

Now suppose to the contrary that $\xi$ is a common root of $A_{n}\left(
z\right)  $ and $A_{n}\left(  -z\right)  .$ Letting $B_{n}\left(  x\right)
=A_{n}\left(  -x\right)  ,$ we have
\[
A_{n}^{\prime\prime}B_{n}-2A_{n}^{\prime}B_{n}^{\prime}+A_{n}B_{n}%
^{\prime\prime}=2\mu\left(  A_{n}^{\prime}B_{n}-A_{n}B_{n}^{\prime}\right)  .
\]
Then by Lemma \ref{b}, either $\xi$ is a repeated root of $A_{n}\left(
z\right)  ,$ or it is a repeated root of $A_{n}\left(  -z\right)  $. This is a contradiction.
\end{proof}

\subsection{Linearization of the symmetric configuration}

Our construction of traveling wave requires that the vortex configuration we
found is nondegenerate in the symmetric setting (in the sense of Lemma
\ref{symmetry}). For small number of vortices, this can be verified directly. To
explain this, we now consider the case of $n=2.$ Let $p_{1},p_{2},p_{3}$ be
the three roots of the Adler-Moser polynomial appeared in Lemma \ref{symmetry}%
. Here $p_{1}$ is the real root and $p_{3}=\bar{p}_{2}.$ Let $q_{i}%
=p_{i}^{\ast}.$ For $z_{1}\in\mathbb{R},z_{2}\in\mathbb{C},$ we define the
force map
\begin{align}
F_{1}\left(  z_{1},z_{2}\right)   &  =\frac{1}{z_{1}-z_{2}}+\frac{1}%
{z_{1}-\bar{z}_{2}}-\frac{1}{2z_{1}}-\frac{1}{z_{1}+z_{2}}-\frac{1}%
{z_{1}-z_{2}^{\ast}},\\
F_{2}\left(  z_{1},z_{2}\right)   &  =\frac{1}{z_{2}-z_{1}}+\frac{1}%
{z_{2}-\bar{z}_{2}}-\frac{1}{z_{2}+z_{1}}-\frac{1}{2z_{2}}-\frac{1}%
{z_{2}-z_{2}^{\ast}}.
\end{align}
We have in mind that $z_{1}$ represents the vortex on the real axis and
$z_{2}$ represents the one lying in the second quadrant. Note that by
symmetry, $F_{1}\left(  z_{1},z_{2}\right)  \in\mathbb{R}.$ Writing
$z_{1}=a_{1},z_{2}=a_{2}+b_{2}i,$ where $a_{i},b_{i}\in\mathbb{R},$ we can
define
\[
F\left(  a_{1},a_{2},b_{2}\right)  :=\left(  F_{1},\operatorname{Re}%
F_{2},\operatorname{Im}F_{2}\right)  .
\]
The configuration $\left(  p_{1},p_{2},p_{3},q_{1},q_{2},q_{3}\right)  $ is
called nondegenerate, if
\[
\det DF\left(  p_{1},\operatorname{Re}p_{2},\operatorname{Im}p_{2}\right)
\neq0.
\]
Numerical computation shows that $\det DF\left(  p_{1},\operatorname{Re}%
p_{2},\operatorname{Im}p_{2}\right)  \neq0.$ Hence it is nondegenerate. It
turns out for $n$ large, this procedure is tedious and we have to find other
ways to overcome this difficulty.

\medskip

In the general case, let $\mathbf{p}=\left(  p_{1},...,p_{n\left(  n+1\right)
/2}\right)  ,\mathbf{q}=\left(  q_{1},...,q_{n\left(  n+1\right)  /2}\right)
.$ Define the map $F:$%
\[
\left(  \mathbf{p},\mathbf{q}\right)  \rightarrow\left(  F_{1},...,F_{n\left(
n+1\right)  /2},G_{1},...,G_{n\left(  n+1\right)  /2}\right)  ,
\]
where
\begin{align}
F_{k}  &  =\sum_{j\neq k}\frac{1}{p_{k}-p_{j}}-\sum_{j}\frac{1}{p_{k}-q_{j}%
},\\
G_{k}  &  =\sum_{j\neq k}\frac{1}{q_{k}-q_{j}}-\sum_{j}\frac{1}{q_{k}-p_{j}}.
\end{align}
Let $a$\thinspace$=\left(  a_{1},...,a_{n\left(  n+1\right)  /2}\right)  $ be
the roots of $A_{n}$ and  $b=-\left(  \bar{a}_{1},...,\bar{a}_{n\left(
n+1\right)  /2}\right)  .$ Moreover, we assume that for $i=1,...,i_{0},$
\[
a_{2i-1}=\bar{a}_{2i},
\]
while for $i=2i_{0}+1,...,n\left(  n+1\right)  /2,$ $\operatorname{Im}%
a_{i}=0.$ Consider the linearization of $F$ at $\left(  \mathbf{p}%
,\mathbf{q}\right)  =\left(  a,b\right)  .$ Denote it by $DF|_{\left(
a,b\right)  }.$ This is a map from $\mathbb{C}^{n\left(  n+1\right)  }$ to
$\mathbb{C}^{n\left(  n+1\right)  }.$

We remark that the points in $a\cup b$ lie \textquotedblleft
approximately\textquotedblright\ on $n$ circles (not exactly on these circles),
and \textquotedblleft approximately\textquotedblright\ on a certain number of
straight lines.

The map $DF|_{\left(  a,b\right)  }$ always has kernel. Indeed, for any
parameter $K=\left(  k_{2},...,k_{n}\right)  ,$ $\Theta_{n}\left(  z,K\right)
$ and $\tilde{\Theta}_{n}\left(  z,K\right)  $ satisfy
\[
\Theta_{n}^{\prime\prime}\tilde{\Theta}_{n}-2\Theta_{n}^{\prime}\tilde{\Theta
}_{n}^{\prime}+\Theta_{n}\tilde{\Theta}_{n}^{\prime\prime}=2\mu\left(
\Theta_{n}^{\prime}\tilde{\Theta}_{n}-\Theta_{n}\tilde{\Theta}_{n}^{\prime
}\right)  .
\]
Differentiating this equation with respect to the translation and the
parameters $k_{i}$,$i=2,...,n,$ we get $n$ (complex) dimensional kernel. Denote
them by $\varpi_{1},...,\varpi_{n}.$

Let $\xi=\left(  \xi_{1},...,\xi_{n\left(  n+1\right)  /2}\right)
\in\mathbb{C}^{n\left(  n+1\right)  /2},$ $\eta=\left(  \eta_{1}%
,...,\eta_{n\left(  n+1\right)  /2}\right)  \in\mathbb{C}^{n\left(
n+1\right)  /2},$ the vector $\left(  \xi,\eta\right)  ,$ with $\eta=\xi
^{\ast}$, is called symmetric if for $i=1,...,i_{0},$
\[
\xi_{2i-1}=\bar{\xi}_{2i},
\]
while for $i=2i_{0}+1,...,n\left(  n+1\right)  /2,$ $\operatorname{Im}\xi
_{i}=0.$

\medskip

The main result of this section is the nondegeneracy of the vortex
configuration given by $A_{n}$, i.e.e all the kernels are given by the above:

\begin{proposition}
\label{sy}Suppose $DF|_{\left(  a,b\right)  }\left(  \xi,\eta\right)  =0$ and
$\left(  \xi,\eta\right)  $ is symmetric. Then $\xi=\eta=0.$
\end{proposition}

Proposition \ref{sy} is proved by the linearization of Darboux transformation and recursive relations.

\medskip

\subsection{Darboux transformation and the nondegeneracy of the configuration}

We first recall the following classical Darboux transformation theorem(Theorem
2.1, \cite{Ma}).

\begin{theorem}
Suppose
\begin{align*}
-\Psi^{\prime\prime}+u\Psi &  =\lambda\Psi,\\
-\Psi_{1}^{\prime\prime}+u\Psi_{1}  &  =\lambda_{1}\Psi_{1}.
\end{align*}
Then the function $\Phi:=W\left(  \Psi_{1},\Psi\right)  /\Psi_{1}$ satisfies
\[
-\Phi^{\prime\prime}+\tilde{u}\Phi=\lambda\Phi,
\]
where $\tilde{u}:=u-2\left(  \ln\Psi_{1}\right)  ^{\prime\prime}.$ The
function $\Phi$ is called the Darboux transformation of $\Psi.$
\end{theorem}

Let $\phi_{n}=\frac{A_{n+1}}{A_{n}}.$ Consider the equation
\begin{equation}
\phi^{\prime\prime}+2\left(  \ln A_{n}\right)  ^{\prime\prime}\phi=0.
\label{Sch}%
\end{equation}
The second order recursive relation $\left(  \ref{anan}\right)  $ is
equivariant to the fact $\phi_{n}$ is solution of $\left(  \ref{Sch}\right)
.$ From \cite{Adler}, we know that $\phi_{n-1}^{-1}$ is also a solution.

On the other hand, we have the following Darboux transformation relation
\[
\left(  2n+1\right)  \phi_{n}^{-1}=\frac{W\left(  \phi_{n-1}^{-1},\phi
_{n}\right)  }{\phi_{n}}.
\]
Indeed, this is equivalent to the relation
\[
\left(  2n+1\right)  A_{n}^{2}=A_{n-1}A_{n+1}^{\prime}-A_{n-1}^{\prime}%
A_{n+1}.
\]
Note that the constant $2n+1$ makes the coefficient of the leading order term
of $A_{n}$ to be $1.$ We also have the reversed transformation
\[
\left(  2n+3\right)  \phi_{n}=\frac{W\left(  \phi_{n}^{-1},\phi_{n+1}\right)
}{\phi_{n}^{-1}}.
\]
Indeed, this is equivalent to
\[
\left(  2n+3\right)  A_{n+1}^{2}=A_{n}A_{n+2}^{\prime}-A_{n}^{\prime}A_{n+2}.
\]

We now recall that the function $\psi_{n}\left(  z\right)  =\frac{B_{n}}%
{A_{n}}e^{\mu z}$ satisfying
\[
\psi_{n}^{\prime\prime}+2\left(  \ln A_{n}\right)  ^{\prime\prime}\psi_{n}%
=\mu^{2}\psi_{n}.
\]
Note that
\[
\psi_{n}=\frac{W\left(  \theta_{1},...,\theta_{2n-1},e^{\mu z}\right)
}{W\left(  \theta_{1},...,\theta_{2n-1}\right)  }.
\]
Then the Darboux transformation\cite{Ma} between $\psi_{n}$ and $\psi_{n+1}$
is given by
\begin{equation}
\psi_{n+1}=\frac{W\left(  \psi_{n},\phi_{n}\right)  }{\phi_{n}}. \label{Dar}%
\end{equation}
Explicitly,
\begin{equation}
\psi_{n+1}=\psi_{n}\left(  \ln\phi_{n}\right)  ^{\prime}-\psi_{n}^{\prime}.
\label{Dar2}%
\end{equation}
Let us verify this for the $n=1$ case. We have%
\begin{align*}
\psi_{1}  &  =\frac{W\left(  \theta_{1},e^{\mu z}\right)  }{W\left(
\theta_{1}\right)  }=\frac{\left(  \mu z-1\right)  e^{\mu z}}{z},\\
\phi_{1}  &  =\frac{A_{2}}{A_{1}}=\frac{z^{3}+k_{2}}{z},
\end{align*}
We also have%
\begin{align*}
\psi_{2}  &  =\frac{W\left(  \theta_{1},\theta_{3},e^{\mu z}\right)  }%
{z^{3}+k_{2}}\\
&  =\frac{e^{\mu z}}{z^{3}+k_{2}}\left(  z-z^{2}\mu+\frac{1}{3}\mu^{2}%
k_{2}+\frac{1}{3}z^{3}\allowbreak\mu^{2}\right)  .
\end{align*}
Therefore,%
\begin{align*}
&  \psi_{1}\left(  \ln\phi_{1}\right)  ^{\prime}-\psi_{1}^{\prime}\\
&  =\frac{\left(  \mu z-1\right)  e^{\mu z}}{z}\frac{d}{dz}\left(  \ln
\frac{z^{3}+k_{2}}{z}\right)  -\frac{d}{dz}\left(  \frac{\left(  \mu
z-1\right)  e^{\mu z}}{z}\right) \\
&  =\psi_{2}.
\end{align*}

Next we would like to analyze the linearized Darboux transformation. First of
all, we linearize the equation $\left(  \ref{anan}\right)  $ at $\left(
A_{n},A_{n+1}\right)  .$ We obtain%
\[
\xi_{n}^{\prime\prime}A_{n+1}-2\xi_{n}^{\prime}A_{n+1}^{\prime}+\xi_{n}%
A_{n+1}^{\prime\prime}+A_{n}^{\prime\prime}\xi_{n+1}-2A_{n}^{\prime}\xi
_{n+1}^{\prime}+A_{n}\xi_{n+1}^{\prime\prime}=0.
\]
Let $\xi_{n}=A_{n}\tilde{\xi}_{n}.$ Then
\begin{align*}
&  \left(  A_{n}^{\prime\prime}\tilde{\xi}_{n}+2A_{n}^{\prime}\tilde{\xi}%
_{n}^{\prime}+A_{n}\tilde{\xi}_{n}^{\prime\prime}\right)  A_{n+1}-2\left(
A_{n}^{\prime}\tilde{\xi}_{n}+A_{n}\tilde{\xi}_{n}^{\prime}\right)
A_{n+1}^{\prime}+A_{n}\tilde{\xi}_{n}A_{n+1}^{\prime\prime}\\
&  +A_{n}^{\prime\prime}A_{n+1}\tilde{\xi}_{n+1}-2A_{n}^{\prime}\left(
A_{n+1}^{\prime}\tilde{\xi}_{n+1}+A_{n+1}\tilde{\xi}_{n+1}^{\prime}\right)
+A_{n}\left(  A_{n+1}^{\prime\prime}\tilde{\xi}_{n+1}+2A_{n+1}^{\prime}%
\tilde{\xi}_{n+1}^{\prime}+A_{n+1}\tilde{\xi}_{n+1}^{\prime\prime}\right) \\
&  =0.
\end{align*}
Introducing $f_{n}=\tilde{\xi}_{n}^{\prime},$ we get%
\begin{align*}
&  A_{n}A_{n+1}f_{n}^{\prime}+\left(  2A_{n}^{\prime}A_{n+1}-2A_{n}%
A_{n+1}^{\prime}\right)  f_{n}\\
&  +A_{n}A_{n+1}f_{n+1}^{\prime}+\left(  2A_{n}A_{n+1}^{\prime}-2A_{n}%
^{\prime}A_{n+1}\right)  f_{n+1}\\
&  =0.
\end{align*}
This equation can be written as
\[
f_{n}^{\prime}+2\left(  \ln\frac{A_{n}}{A_{n+1}}\right)  ^{\prime}%
f_{n}+f_{n+1}^{\prime}+2\left(  \ln\frac{A_{n+1}}{A_{n}}\right)  ^{\prime
}f_{n+1}=0.
\]
Hence for any given function $f_{n+1},$ we can solve this equation and get
\begin{align}
f_{n}  &  =\frac{A_{n+1}^{2}}{A_{n}^{2}}\int_{0}^{z}\frac{A_{n}^{2}}%
{A_{n+1}^{2}}\left(  f_{n+1}^{\prime}+2\left(  \ln\frac{A_{n+1}}{A_{n}%
}\right)  ^{\prime}f_{n+1}\right)  ds\nonumber\\
&  =-f_{n+1}+2\frac{A_{n+1}^{2}}{A_{n}^{2}}\int_{0}^{z}\frac{A_{n}^{2}%
}{A_{n+1}^{2}}f_{n+1}^{\prime}ds. \label{s1}%
\end{align}
The last equality follows from integrating by parts for the second term.

Next, we linearize the equation $\left(  \ref{Dar2}\right)  $ at $\left(
A_{n},A_{n+1}\right)  $ and obtain%
\[
\sigma_{n+1}=\sigma_{n}\left(  \ln\phi_{n}\right)  ^{\prime}-\sigma
_{n}^{\prime}+\psi_{n}\left(  \frac{\xi_{n+1}}{A_{n+1}}-\frac{\xi_{n}}{A_{n}%
}\right)  ^{\prime}.
\]
We recall that $\left(  \frac{\xi_{n}}{A_{n}}\right)  ^{\prime}=f_{n}$. Hence
we get the equation%
\[
\sigma_{n}^{\prime}-\sigma_{n}\left(  \ln\phi_{n}\right)  ^{\prime}=\psi
_{n}\left(  f_{n+1}-f_{n}\right)  -\sigma_{n+1}.
\]
From this we get
\begin{equation}
\sigma_{n}=\phi_{n}\int_{0}^{z}\phi_{n}^{-1}\left(  \psi_{n}\left(
f_{n+1}-f_{n}\right)  -\sigma_{n+1}\right)  ds. \label{s2}%
\end{equation}
We are lead to the system
\begin{equation}
\left\{
\begin{array}
[c]{l}%
f_{n}^{\prime}+2\left(  \ln\frac{A_{n}}{A_{n+1}}\right)  ^{\prime}%
f_{n}+f_{n+1}^{\prime}+2\left(  \ln\frac{A_{n+1}}{A_{n}}\right)  ^{\prime
}f_{n+1}=0,\\
\sigma_{n}^{\prime}-\sigma_{n}\left(  \ln\phi_{n}\right)  ^{\prime}=\psi
_{n}\left(  f_{n+1}-f_{n}\right)  -\sigma_{n+1}.
\end{array}
\right.  \label{s}%
\end{equation}
For given function $f_{n+1}$ and $\sigma_{n+1},$ we can solve this system and
get a solution $\left(  f_{n},\sigma_{n}\right)  $ from $\left(
\ref{s1}\right)  ,\left(  \ref{s2}\right)  .$

Let $K_{0}=\left(  -\frac{1}{2},-\frac{1}{2},....\right)  .$ For each fixed
$n,$ and $j=2,...,n,$ we define the polynomials
\begin{align*}
\omega_{n,j}  &  =\left(  \frac{\partial_{k_{j}}\Theta_{n}\left(
z+1,K\right)  |_{K=K_{0}}}{A_{n}}\right)  ^{\prime},\\
\tilde{\omega}_{n,j}  &  =\left(  \frac{\partial_{k_{j}}\tilde{\Theta}%
_{n}\left(  z+1,K\right)  |_{K=K_{0}}}{B_{n}}\right)  ^{\prime}.
\end{align*}
Let
\[
\beta_{n,j}=e^{\mu z}\left(  -\frac{B_{n}\omega_{n,j}}{A_{n}^{2}}+\frac
{\tilde{\omega}_{n,j}}{A_{n}}\right)  .
\]

\begin{lemma}
$f_{n}=\omega_{n,j},\sigma_{n}=\beta_{n,j}$ satisfy the system $\left(
\ref{s}\right)  .$
\end{lemma}

We also need the following uniqueness lemma on the symmetric configuration.

\begin{lemma}
Suppose $\hat{K}$ is a $n-1$ dimensional vector and $\left\vert \hat
{K}-K\right\vert +t+\frac{1}{2}<\delta$ for some small $\delta,$ with $\hat
{K}\neq K.$ Then
\[
\Theta_{n}\left(  -z-t,\hat{K}\right)  \neq\tilde{\Theta}_{n}\left(
z-t,\hat{K}\right)  .
\]

\end{lemma}

\begin{proof}
We prove this by induction. This is true for $n=1.$ Suppose this is true for
$n=k,$ we prove that it is also true for $n=k+1.$ Indeed, suppose to be
contrary that $\Theta_{n}\left(  -z-t,\hat{K}\right)  \neq\tilde{\Theta}%
_{n}\left(  z-t,\hat{K}\right)  .$
\begin{align*}
&  \Theta_{n}^{\prime\prime}\left(  z-t,\hat{K}\right)  \Theta_{n-1}\left(
z-t,\hat{K}\right)  -2\Theta_{n}^{\prime}\left(  z-t,\hat{K}\right)
\Theta_{n-1}^{\prime}\left(  z-t,\hat{K}\right) \\
&  +\Theta_{n}\left(  z-t,\hat{K}\right)  \Theta_{n-1}^{\prime\prime}\left(
z-t,\hat{K}\right) \\
&  =0.
\end{align*}
Replacing $z$ be $-z,$ we get
\begin{align*}
&  \tilde{\Theta}_{n}^{\prime\prime}\left(  z-t,\hat{K}\right)  \Theta
_{n-1}\left(  -z-t,\hat{K}\right)  -2\tilde{\Theta}_{n}^{\prime}\left(
z-t,\hat{K}\right)  \Theta_{n-1}^{\prime}\left(  -z-t,\hat{K}\right) \\
&  +\tilde{\Theta}_{n}\left(  z-t,\hat{K}\right)  \Theta_{n-1}^{\prime\prime
}\left(  -z-t,\hat{K}\right) \\
&  =0.
\end{align*}
This then implies that
\[
\Theta_{n-1}\left(  -z-t,\hat{K}\right)  =\tilde{\Theta}_{n-1}\left(
z-t,\hat{K}\right)  .
\]
Hence $t=-\frac{1}{2},$ and the first $n-2$ components of $\hat{K}$ is
$-\frac{1}{2}.$ It then follows that the last component of $\hat{K}$ is also
$-\frac{1}{2}.$ This finishes the proof.
\end{proof}

Since $\psi_{n}=\frac{B_{n}}{A_{n}}e^{\mu z},$ we have the relation%
\[
\sigma_{n}e^{-\mu z}=-\frac{B_{n}\xi_{n}}{A_{n}^{2}}+\frac{\eta_{n}}{A_{n}}.
\]
Hence the function $\eta_{n}$ is given in terms of $\sigma_{n},\xi_{n}$ by
\begin{equation}
\eta_{n}=A_{n}\sigma_{n}e^{-\mu z}+\frac{B_{n}}{A_{n}}\xi_{n}. \label{yita}%
\end{equation}
Note that $\left(  A_{n},B_{n}\right)  $ satisfies
\[
A_{n}^{\prime\prime}B_{n}-2A_{n}^{\prime}B_{n}^{\prime}+A_{n}B_{n}%
^{\prime\prime}-2\mu\left(  A_{n}^{\prime}B_{n}-A_{n}B_{n}^{\prime}\right)
=0.
\]
Linearizing this equation we get
\begin{align*}
&  \xi_{n}^{\prime\prime}B_{n}-2\xi_{n}^{\prime}B_{n}^{\prime}+\xi_{n}%
B_{n}^{\prime\prime}-2\mu\left(  \xi_{n}^{\prime}B_{n}-\xi_{n}B_{n}^{\prime
}\right) \\
&  +A_{n}^{\prime\prime}\eta_{n}-2A_{n}^{\prime}\eta_{n}^{\prime}+A_{n}%
\eta_{n}^{\prime\prime}-2\mu\left(  A_{n}^{\prime}\eta_{n}-A_{n}\eta
_{n}^{\prime}\right) \\
&  =0.
\end{align*}
In the case $n=0,$ we have $A_{n}=B_{n}=1,$ the above equation reads
\begin{equation}
\xi_{0}^{\prime\prime}-2\mu\xi_{0}^{\prime}+\eta_{0}^{\prime\prime}+2\mu
\eta_{0}^{\prime}=0. \label{xi}%
\end{equation}
That is, $\left(  \xi_{0}+\eta_{0}\right)  ^{\prime}=2\mu\left(  \xi_{0}%
-\eta_{0}\right)  .$ By $\left(  \ref{yita}\right)  ,$%
\[
\eta_{0}=\sigma_{0}e^{-\mu z}+\xi_{0}.
\]
Hence
\begin{align*}
\xi_{0}+\eta_{0}  &  =\sigma_{0}e^{-\mu z}+2\xi_{0},\\
\xi_{0}-\eta_{0}  &  =-\sigma_{0}e^{-\mu z}.
\end{align*}
It follows that
\begin{equation}
\left(  \sigma_{0}e^{-\mu z}+2\xi_{0}\right)  ^{\prime}+2\mu\left(  \sigma
_{0}e^{-\mu z}\right)  =0. \label{sigma}%
\end{equation}

Now suppose the Adler-Moser polynomial $A_{N}$ satisfies assumption (A). Given
functions $\xi_{N}$ and $\eta_{N},$ we have corresponding functions
$f_{N},\sigma_{N}.$ Using $\left(  \ref{s}\right)  ,$ we can define
recursively $\left(  \xi_{N-1},\eta_{N-1},f_{N-1},\sigma_{N-1}\right)
,...,\left(  \xi_{0},\eta_{0},f_{0},\sigma_{0}\right)  .$ Linearizing the
Darboux transformation, we find that $\xi_{0},\eta_{0}$ satisfy $\left(
\ref{xi}\right)  .$

\begin{proposition}
\label{li}Suppose $\xi_{N},\eta_{N}$ are polynomials with degree less than
$N\left(  N+1\right)  /2.$ Then $\xi_{0}=\eta_{0}=0.$
\end{proposition}

\begin{proof}
We first consider the case that for any $n\neq j\leq N,$ $A_{n}$ and $A_{j}$
have no common roots. (This assumption is true for $N=34$, as can be verified by
Maple.)  The idea for the general case is similar but notations are more involved.

Since $f_{N}=\left(  \frac{\xi_{N}}{A_{N}}\right)  ^{\prime}$, $f_{N}$ is a
rational function with possible poles at the roots of $A_{N}.$ We know that
for each $n\leq N-1,$ $f_{n}$ and $f_{n+1}$ are related by
\begin{equation}
f_{n}=-f_{n+1}+2\frac{A_{n+1}^{2}}{A_{n}^{2}}\int_{0}^{z}\frac{A_{n}^{2}%
}{A_{n+1}^{2}}f_{n+1}^{\prime}ds. \label{fn}%
\end{equation}
Hence $f_{n}$ has possible poles at the roots of $A_{n},A_{n+1},...,A_{N}.$ We
remark that as a complex valued function with poles, $f_{n}$ may be
multiple-valued. By $\left(  \ref{fn}\right)  ,$
\begin{align}
f_{n}-f_{n+1}  &  =-2f_{n+1}+2\frac{A_{n+1}^{2}}{A_{n}^{2}}\int_{0}^{z}%
\frac{A_{n}^{2}}{A_{n+1}^{2}}f_{n+1}^{\prime}ds\nonumber\\
&  =-2\frac{A_{n+1}^{2}}{A_{n}^{2}}\int_{0}^{z}f_{n+1}\left(  \frac{A_{n}^{2}%
}{A_{n+1}^{2}}\right)  ^{\prime}ds. \label{fnn}%
\end{align}
In particular,
\begin{equation}
f_{0}=f_{1}+4\left(  z-1\right)  ^{2}\int_{0}^{z}\frac{f_{1}}{\left(
s-1\right)  ^{3}}ds. \label{ff}%
\end{equation}
On the other hand,
\[
\sigma_{n}=\phi_{n}\int_{0}^{z}\phi_{n}^{-1}\left(  \psi_{n}\left(
f_{n+1}-f_{n}\right)  -\sigma_{n+1}\right)  ds.
\]
Recall that $\phi_{0}=z-1,\psi_{0}=1.$ Hence
\begin{align*}
\sigma_{0}  &  =\left(  z-1\right)  \int_{0}^{z}\frac{1}{z-1}e^{\mu z}\left(
f_{1}-f_{0}-\sigma_{1}\right)  ds\\
&  =\frac{A_{n+1}}{A_{n}}\int_{0}^{z}\frac{B_{n}}{A_{n+1}}e^{\mu s}\left(
f_{n+1}-f_{n}\right)  ds-\frac{A_{n+1}}{A_{n}}\int_{0}^{z}\frac{A_{n}}%
{A_{n+1}}\sigma_{n+1}ds.
\end{align*}
By $\left(  \ref{sigma}\right)  ,$%
\[
\left(  \sigma_{0}e^{\mu z}\right)  ^{\prime}+2e^{2\mu z}f_{0}=0.
\]
Hence using the fact that $\mu=1,$ we obtain
\begin{align}
&  \left(  \sigma_{0}e^{\mu z}\right)  ^{\prime}+2e^{2\mu z}f_{0}\nonumber\\
&  =e^{2\mu z}\left(  f_{0}+f_{1}-\sigma_{1}\right) \nonumber\\
&  +ze^{\mu z}\int_{0}^{z}\frac{1}{s-1}e^{\mu s}\left(  f_{1}-f_{0}-\sigma
_{1}\right)  ds\nonumber\\
&  =0. \label{f01}%
\end{align}

Our next aim is to show that $f_{1}$ has no singularity except the point
$z=1.$ Let us consider the term $\frac{A_{n+1}}{A_{n}}\int_{0}^{z}\frac{A_{n}%
}{A_{n+1}}\sigma_{n+1}ds.$ Let $z=d_{0}$ be a singularity of $f_{n}$ which is
not the root of $A_{n}.$ Then loosely speaking, the degree of singularity
$\sigma_{n}$ is smaller than that of $\sigma_{n+1}$ and $f_{n}.$ By $\left(
\ref{fnn}\right)  ,$ $f_{n}$ and $f_{n+1}$ has essentially the same degree of
singularity at $d_{0}.$ But this contradicts with the identity $\left(
\ref{f01}\right)  .$ Hence $f_{1}$ can only have singularity at $z=1.$

Now we show that $f_{0}=0.$ To see this, we observe that since $f_{1}$ has no
other singularities, by the recursive relation, we deduce that $f_{1}$ is
actually single valued and $f_{1}=c_{1}\frac{1}{z-1}+c_{2}\frac{1}{\left(
z-1\right)  ^{2}},$ and
\[
\sigma_{1}=\phi_{1}\int_{0}^{z}\phi_{1}^{-1}\left(  \psi_{1}\left(
f_{2}-f_{1}\right)  -\sigma_{2}\right)  ds.
\]
Putting this into $\left(  \ref{f01}\right)  ,$ we find that $c_{1}=0.$ Hence
$f_{0}=0$ and $\sigma_{0}=0.$
\end{proof}

Now we can prove Proposition \ref{sy}. By Proposition \ref{li}, the kernel of
the map $DF|_{\left(  a,b\right)  }$ is given by linear combination of
$\varpi_{i}.$ For $\mu=1,$ $k_{i}=-\frac{1}{2}$ are the only parameters for
which $\Theta_{n}$ and $\tilde{\Theta}_{n}$ give arise to symmetric
configuration. Hence the configuration is nondegenerate.

\setcounter{equation}{0}
\section{\label{Pri}Preliminaries on the Ginzburg-Landau equation}

In this section, we recall some results on the Ginzburg-Landau equation. Most
of the materials in this section can be found in \cite{Pacard} (possibly with
different notations though).

Stationary solutions of the GP equation $\left(  \ref{GP}\right)  $ solve the
following Ginzburg-Landau equation
\begin{equation}
-\Delta\Phi=\Phi\left(  1-\left\vert \Phi\right\vert ^{2}\right)  \text{ in
}\mathbb{R}^{2}, \label{GL}%
\end{equation}
where $\Phi$ is a complex valued function.  As we mentioned before, it has
degree $\pm d$ vortices of the form $S_{d}\left(  r\right)  e^{\pm id\theta}.$
The asymptotic behavior of $S_{d}$ can be described. It is known that as
$r\rightarrow+\infty,$%
\begin{equation}
S_{d}\left(  r\right)  =1-\frac{d^{2}}{2r^{2}}+O\left(  r^{-4}\right)  .
\label{asy}%
\end{equation}
On the other hand, as $r\rightarrow0,$ there is a constant $\kappa=\kappa
_{d}>0$ such that
\begin{equation}
S_{d}\left(  r\right)  =\kappa r\left(  1-\frac{r^{2}}{8}+O\left(
r^{4}\right)  \right)  . \label{near 0}%
\end{equation}
See \cite{Fife} for the proof of these facts.

Let $\varepsilon>0$ be small. For technical reasons, we need to modify $S$ in
the region $r>C_{0}\varepsilon^{-1},$ where $C_{0}$ is a fixed large constant,
such that $S\left(  r\right)  =1$ for $r>C_{0}\varepsilon^{-1}+1.$ We still
denote it by $S$ for notational simplicity.

The linearized operator of the Ginzburg-Landau equation around $v_{+}$ will be
denoted by $L:$
\[
\eta\rightarrow\Delta\eta+\left(  1-\left\vert v_{+}\right\vert ^{2}\right)
\eta-2v_{+}\operatorname{Re}\left(  \eta\bar{v}_{+}\right)  .
\]
It turns out to be more convenient to study the operator
\[
\mathcal{L}\eta:=e^{-i\theta}L\left(  e^{i\theta}\eta\right)  .
\]
If we write the complex function $\eta$ as $w_{1}+iw_{2}$ with $w_{1},w_{2}$
being real valued functions, then explicitly
\begin{align*}
\mathcal{L}\eta &  =e^{-i\theta}\Delta\left(  e^{i\theta}\eta\right)  +\left(
1-S^{2}\right)  \eta-2S^{2}w_{1}\\
&  =\Delta w_{1}+\left(  1-3S^{2}\right)  w_{1}-\frac{1}{r^{2}}w_{1}-\frac
{2}{r^{2}}\partial_{\theta}w_{2}\\
&  +i\left(  \Delta w_{2}+\left(  1-S^{2}\right)  w_{2}-\frac{1}{r^{2}}%
w_{2}+\frac{2}{r^{2}}\partial_{\theta}w_{1}\right)  .
\end{align*}
Invariance of the equation $\left(  \ref{GL}\right)  $ under rotation and
translation gives us three linearly independent kernels of the operator
$\mathcal{L}$, called Jacobi fields. Rotational invariance yields the
solution
\begin{equation}
\Phi^{0}:=ie^{-i\theta}v^{+}=iS, \label{f}%
\end{equation}
while the translational invariance along $x$ and $y$ direction leads to the
solutions
\begin{align*}
\Phi^{+1}  &  :=S^{\prime}\cos\theta-\frac{S}{r}\sin\theta,\\
\Phi^{-1}  &  :=S^{\prime}\sin\theta+\frac{S}{r}\cos\theta.
\end{align*}
Note that these kernels are bounded but decay slowly at infinity, hence not in
$L^{2}\left( {\mathbb R}^{2}\right)  .$ As a consequence, the analysis of the mapping
property of $\mathcal{L}$ is quite delicate. An important fact is that $v_{+}$
is nondegenerate in the sense of all the bounded solutions of $\mathcal{L}%
\eta=0$ are given by linear combinations of $\Phi^{0}$ and $\Phi^{+},\Phi^{-}$. (See [Theorem 3.2, \cite{Pacard}]. Another proof can be found in \cite{DFK}.) Similar results hold for the degree $-1$
vortex $v_{-}.$ It is worth mentioning that the nondegeneracy of those higher
degree vortices $e^{id\theta}S_{d}\left(  r\right)  ,$ $\left\vert
d\right\vert >1,$ is still an open problem. Actually this is the main reason
why we only deal with the degree $\pm1$ vortices in this paper. 

\medskip

The analysis of the asymptotic behavior of the kernels of $\mathcal{L}$ near
$0$ and $\infty$ is crucial in understanding the mapping property of the
linearized operator $\mathcal{L}$. In doing this, the main strategy is to
decompose the kernel into different Fourier modes. Let us now briefly describe
the results in the sequel. Lemma \ref{n=0}, Lemma \ref{n=1} and Lemma
\ref{n=2} below can be found in Section 3.3 of \cite{Pacard}.

We start the discussion with the lowest Fourier mode, which is the simplest case.

\begin{lemma}
\label{n=0}Suppose $a$ is a complex valued solution of the equation
$\mathcal{L}a=0,$ depending only on $r.$ \newline(I) As $r\rightarrow0,$
either $\left\vert a\right\vert $ blows up at least like $r^{-1},$ or $a$ can
be written as a linear combination of two linearly independent solutions
$w_{0,1},w_{0,2},$ with
\begin{align*}
w_{0,1}\left(  r\right)   &  =r\left(  1+O\left(  r^{2}\right)  \right)  ,\\
w_{0,2}\left(  r\right)   &  =ir\left(  1+O\left(  r^{2}\right)  \right)  .
\end{align*}
\newline(II) As $r\rightarrow+\infty,$ if $a$ is an imaginary valued function,
then $a=c_{1}+c_{2}\ln r+O\left(  r^{-2}\right)  ;$ if $a$ is real valued,
then it either blows up or decays exponentially.
\end{lemma}

\begin{proof}
We sketch the proof for completeness.

If $\mathcal{L}a=0$ and the complex function $a$ depends only on $r,$ then $a$
will satisfy
\begin{equation}
a^{\prime\prime}+\frac{1}{r}a^{\prime}-\frac{1}{r^{2}}a=S^{2}\bar{a}-\left(
1-2S^{2}\right)  a. \label{F0}%
\end{equation}
Note that this equation is not complex linear and its solution space is a
$4$-dimensional real vector space. The Jacobi field $\Phi^{0}$ defined by
$\left(  \ref{f}\right)  $ is a purely imaginary solution of $\left(
\ref{F0}\right)  .$ Writing $a=a_{1}+a_{2}i,$ where $a_{i}$ are real valued
functions, we get from $\left(  \ref{F0}\right)  $ two \textit{decoupled}
equations:%
\[
a_{1}^{\prime\prime}+\frac{1}{r}a_{1}^{\prime}-\frac{1}{r^{2}}a_{1}+\left(
1-3S^{2}\right)  a_{1}=0,
\]%
\begin{equation}
a_{2}^{\prime\prime}+\frac{1}{r}a_{2}^{\prime}-\frac{1}{r^{2}}a_{2}+\left(
1-S^{2}\right)  a_{2}=0. \label{Ia}%
\end{equation}
Observe that due to $\left(  \ref{asy}\right)  $, as $r\rightarrow+\infty,$
\begin{align*}
1-3S^{2}-r^{-2}  &  =-2+O\left(  r^{-2}\right)  ,\\
1-S^{2}-r^{-2}  &  =O\left(  r^{-4}\right)  .
\end{align*}
While due to $\left(  \ref{near 0}\right)  ,$ as $r\rightarrow0,$
$1-S^{2}=1+O\left(  r^{2}\right)  .$ The results of this lemma then follow
from a perturbation argument.
\end{proof}

For each integer $n\geq1,$ we consider kernels of $\mathcal{L}$ the form
$a\left(  r\right)  e^{in\theta}+b\left(  r\right)  e^{-in\theta}.$ The
complex valued functions $a,b$ will satisfy the following coupled ODE system
in $\left(  0,+\infty\right)  :$%
\begin{equation}
\left\{
\begin{array}
[c]{c}%
a^{\prime\prime}+\frac{1}{r}a^{\prime}-\frac{\left(  n+1\right)  ^{2}}{r^{2}%
}a=S^{2}\bar{b}-\left(  1-2S^{2}\right)  a\\
b^{\prime\prime}+\frac{1}{r}b^{\prime}-\frac{\left(  n-1\right)  ^{2}}{r^{2}%
}b=S^{2}\bar{a}-\left(  1-2S^{2}\right)  b.
\end{array}
\right.  \label{sys}%
\end{equation}
By analyzing this coupled ODE system, one gets the precise asymptotic behavior
of its solutions. The next lemma deals with the $n=1$ case.

\begin{lemma}
\label{n=1}Suppose $w=a\left(  r\right)  e^{in\theta}+b\left(  r\right)
e^{-in\theta}$ solves $\mathcal{L}w=0.$ \newline(I) As $r\rightarrow0,$ either
$\left\vert w\right\vert $ blows up at least like $-\ln r,$ or $w$ can be
written as a linear combination of $4$ linearly independent solutions
$w_{1,i},i=1,...,4$, satisfying: As $r\rightarrow0,$
\begin{align*}
w_{1,1}  &  =r^{2}\left(  1+O\left(  r^{2}\right)  \right)  e^{i\theta
}+O\left(  r^{6}\right)  e^{-i\theta},\\
w_{1,2}  &  =ir^{2}\left(  1+O\left(  r^{2}\right)  \right)  e^{i\theta
}+O\left(  r^{6}\right)  e^{-i\theta},\\
w_{1,3}  &  =\left(  1+O\left(  r^{2}\right)  \right)  e^{-i\theta}+O\left(
r^{4}\right)  e^{i\theta},\\
w_{1,4}  &  =i\left(  1+O\left(  r^{2}\right)  \right)  e^{-i\theta}+O\left(
r^{4}\right)  e^{i\theta}.
\end{align*}
\newline(II) As $r\rightarrow+\infty,$ either $\left\vert w\right\vert $ is
unbounded(blows up exponentially or like $r$), or $\left\vert w\right\vert $
decays to zero(exponentially or like $r^{-1}$).
\end{lemma}

For the $n\geq2$ case, we have the following

\begin{lemma}
\label{n=2}Suppose $w=a\left(  r\right)  e^{i\theta}+b\left(  r\right)
e^{-i\theta}$ solves $\mathcal{L}w=0.$ \newline(I) As $r\rightarrow0,$ either
$\left\vert w\right\vert $ blows up at least like $r^{1-n},$ or $w$ can be
written as a linear combination of $4$ linearly independent solutions
$w_{1,i},i=1,...,4$, satisfying: As $r\rightarrow0,$
\begin{align*}
w_{n,1}  &  =r^{n+1}\left(  1+O\left(  r^{2}\right)  \right)  e^{in\theta
}+O\left(  r^{n+5}\right)  e^{-i\theta},\\
w_{n,2}  &  =ir^{n+1}\left(  1+O\left(  r^{2}\right)  \right)  e^{in\theta
}+O\left(  r^{n+5}\right)  e^{-i\theta},\\
w_{n,3}  &  =r^{n-1}\left(  1+O\left(  r^{2}\right)  \right)  e^{-in\theta
}+O\left(  r^{n+3}\right)  e^{in\theta},\\
w_{n,4}  &  =ir^{n-1}\left(  1+O\left(  r^{2}\right)  \right)  e^{-in\theta
}+O\left(  r^{n+3}\right)  e^{in\theta}.
\end{align*}
\newline(II) As $r\rightarrow+\infty,$ either $\left\vert w\right\vert $ is
unbounded(blows up exponentially or like $r^{n}$), or $\left\vert w\right\vert
$ decays to zero(exponentially or like $r^{-n}$).
\end{lemma}

By Lemma \ref{n=2}, for $n\geq3,$ if $\mathcal{L}w=0$ and $w$ is bounded near
$0,$ then decays at least like $r^{2}$ as $r\rightarrow0,$ hence decaying
faster than the vortex solution itself. For $n\leq2,$ solutions of
$\mathcal{L}w=0$ bounded near $0$ behaves like $O\left(  r\right)  $ or
$O\left(  1\right)  .$ Note that $\Phi_{0},\Phi_{+1},\Phi_{-1}$ have this
property. Now define $\Psi_{0}=\kappa w_{0,2},$
\begin{align*}
\Psi_{+1}  &  =\kappa w_{1,3}+\frac{\kappa}{8}w_{1,1},\Psi_{-1}=\kappa
w_{1,4}-\frac{\kappa}{8}w_{1,2},\\
\Psi_{+2}  &  =w_{2,3},\Psi_{-2}=w_{2,4}.
\end{align*}
They behave like $O\left(  r\right)  $ or $O\left(  1\right)  $ near $0,$ but
blows up as $r\rightarrow+\infty.$

\bigskip

\setcounter{equation}{0}
\section{\label{non}Construction of multi-vortex solutions}

\subsection{Approximate solutions and estimate of the error}

We would like to construct traveling wave solutions by gluing together
$N\left(  N+1\right)  /2$ pairs of degree $\pm1$ vortices. Let us simply
choose $N=2,$ the general case is similar, but notations will be much more involved.

For $k=1,2,3,$ let $p_{k},q_{k}\in\mathbb{C}$. We have in mind that $p_{k}$
are close to roots of the Adler-Moser polynomial $A_{2}.$ We define the
translated vortices
\[
u_{k}=v_{+}\left(  z-\varepsilon^{-1}p_{k}\right)  ,u_{3+k}=v_{-}\left(
z-\varepsilon^{-1}q_{k}\right)  .
\]
We then define the approximate solution
\[
u:=\prod\limits_{k=1}^{6}u_{k}.
\]
Note that as $r\rightarrow+\infty,$ $u\rightarrow1.$ Hence the degree of $u$
is $0.$ Let us denote the function $z\rightarrow\overline{u\left(  z\right)
}$ by $\bar{u}.$ The next lemma states that the real part of $u$ is even both
in the $x$ and $y$ variables, while the imaginary part is even in $x$ and odd
in $y.$

\begin{lemma}
The approximate solution $u$ has the following symmetry:%
\[
u\left(  \bar{z}\right)  =\bar{u}\left(  z\right)  ,\text{ }u\left(  z^{\ast
}\right)  =u\left(  z\right)  .
\]

\end{lemma}

\begin{proof}
Observe that the standard vortex $v_{+}=S\left(  r\right)  e^{i\theta}$
satisfies
\[
v_{+}\left(  \bar{z}\right)  =\bar{v}_{+}\left(  z\right)  ,v_{+}\left(
z^{\ast}\right)  =\left(  v_{+}\left(  z\right)  \right)  ^{\ast}.
\]
The oppositive vortex $v_{-}$ has similar properties. Hence using the fact
that the set $\left\{  p_{1},p_{2},p_{3}\right\}  $ is invariant with respect
to the reflection across the $x$ axis, we get
\begin{align*}
u\left(  \bar{z}\right)   &  =\prod\limits_{k=1}^{3}\left(  v_{+}\left(
\bar{z}-\varepsilon^{-1}p_{k}\right)  v_{-}\left(  \bar{z}-\varepsilon
^{-1}q_{k}\right)  \right) \\
&  =\prod\limits_{k=1}^{3}\left(  \bar{v}_{+}\left(  z-\varepsilon^{-1}\bar
{p}_{k}\right)  \bar{v}_{-}\left(  z-\varepsilon^{-1}\bar{q}_{k}\right)
\right)  =\bar{u}\left(  z\right)  .
\end{align*}
Moreover, since $v_{-}=\bar{v}_{+},$ we have
\begin{align*}
u\left(  z^{\ast}\right)   &  =\prod\limits_{k=1}^{3}\left(  v_{+}\left(
z^{\ast}-\varepsilon^{-1}p_{k}\right)  v_{-}\left(  z^{\ast}-\varepsilon
^{-1}q_{k}\right)  \right) \\
&  =\prod\limits_{k=1}^{3}\left(  \left(  v_{+}\left(  z-\varepsilon^{-1}%
q_{k}\right)  \right)  ^{\ast}\left(  v_{-}\left(  z-\varepsilon^{-1}%
p_{k}\right)  \right)  ^{\ast}\right) \\
&  =\prod\limits_{k=1}^{3}\left(  \bar{v}_{+}\left(  z-\varepsilon^{-1}%
q_{k}\right)  \left(  \bar{v}_{-}\left(  z-\varepsilon^{-1}p_{k}\right)
\right)  \right)  =u\left(  z\right)  .
\end{align*}
This finishes the proof.
\end{proof}

We use $E\left(  u\right)  $ to denote the error of the approximate solution:%
\[
E\left(  u\right)  :=\varepsilon i\partial_{y}u+\Delta u+u\left(  1-\left\vert
u\right\vert ^{2}\right)  .
\]
We have
\begin{align*}
\Delta u  &  =\Delta\left(  u_{1}...u_{6}\right) \\
&  =\sum\limits_{k}\left(  \Delta u_{k}\prod\limits_{j\neq k}u_{j}\right)
+\sum\limits_{k\neq j}\left(  \left(  \nabla u_{k}\star\nabla u_{j}\right)
\prod\limits_{l\neq k,j}u_{l}\right)  ,
\end{align*}
where symbol $\star$ denotes $\nabla u_{k}\star\nabla u_{j}:=\partial
_{x}u_{k}\partial_{x}u_{j}+\partial_{y}u_{k}\partial_{y}u_{j}.$ On the other
hand, writing $\left\vert u_{k}\right\vert ^{2}-1=\rho_{k},$ we obtain
\[
\left\vert u\right\vert ^{2}-1=\prod\limits_{k}\left(  1+\rho_{k}\right)
-1=\sum_{k}\rho_{k}+\sum_{k=2}^{6}Q_{k},
\]
where $Q_{k}=\sum_{i_{1}<i_{2}<\cdot\cdot\cdot<i_{k}}\left(  \rho_{i_{1}}%
\cdot\cdot\cdot\rho_{i_{k}}\right)  .$ Using the fact that $u_{k}$ solves the
Ginzburg-Landau equation, we get
\begin{align*}
E\left(  u\right)   &  =\varepsilon i\sum\limits_{k}\left(  \partial_{y}%
u_{k}\prod\limits_{j\neq k}u_{j}\right) \\
&  +\sum\limits_{k,j,k\neq j}\left(  \left(  \nabla u_{k}\star\nabla
u_{j}\right)  \prod\limits_{l\neq i,j}u_{l}\right)  -u\sum_{k=2}^{6}Q_{k}.
\end{align*}
We have in mind that the main order terms are $\partial_{y}u_{k}%
\prod\limits_{j\neq k}u_{j}$ and $\left(  \nabla u_{k}\star\nabla
u_{j}\right)  \prod\limits_{l\neq k,j}u_{l}.$

Throughout the paper $\left(  r_{j},\theta_{j}\right)  $ will denote the polar
coordinate with respect to the point $p_{j}.$ Note that
\[
\partial_{x}\left(  e^{i\theta}\right)  =-\frac{yie^{i\theta}}{r^{2}}%
,\partial_{y}\left(  e^{i\theta}\right)  =\frac{xie^{i\theta}}{r^{2}}.
\]
We compute
\begin{align*}
\partial_{x}u_{k}  &  =-\frac{iy_{k}e^{i\theta_{k}}}{r_{k}^{2}}S\left(
r_{k}\right)  +\frac{x_{k}}{r_{k}}S^{\prime}\left(  r_{k}\right)
e^{i\theta_{k}},\\
\partial_{y}u_{k}  &  =\frac{ix_{k}e^{i\theta_{k}}}{r_{k}}S\left(
r_{k}\right)  +\frac{y_{k}}{r_{k}}S^{\prime}\left(  r_{k}\right)
e^{i\theta_{k}}.
\end{align*}

\begin{lemma}
In the region $\left\vert z\right\vert >C\varepsilon^{-1},$
\begin{align*}
\left\vert E\left(  u\right)  \right\vert  &  \leq C\left\vert z\right\vert
^{-2},\\
\operatorname{Im}\left(  E\left(  u\right)  \right)   &  \leq C\left\vert
z\right\vert ^{-3}.
\end{align*}

\end{lemma}

\begin{proof}
We first estimate, for $z>C\varepsilon^{-1}$,%
\begin{align*}
\left\vert \theta_{1}-\theta_{4}\right\vert  &  \leq Cr^{-1},\\
\left\vert \partial_{x}\left(  \theta_{1}-\theta_{4}\right)  \right\vert  &
\leq Cr^{-2}.
\end{align*}
Hence $\left\vert \partial_{x}u\right\vert \leq Cr^{-2}.$ Next,
\begin{align*}
\left\vert \nabla u_{k}\star\nabla u_{j}\right\vert  &  =\left\vert
\partial_{x}u_{k}\partial_{x}u_{j}+\partial_{y}u_{k}\partial_{y}%
u_{j}\right\vert \\
&  \leq\left\vert \partial_{x}u_{k}\right\vert \left\vert \partial_{x}%
u_{j}\right\vert +\left\vert \partial_{y}u_{k}\right\vert \left\vert
\partial_{y}u_{j}\right\vert \\
&  \leq C\left\vert z\right\vert ^{-2}.
\end{align*}
Finally, since $\rho_{k}\leq C\left\vert z\right\vert ^{-2},$ we have
$Q_{k}\leq C\left\vert z\right\vert ^{-4}.$ This finishes the proof.
\end{proof}

\subsection{Projection of the error on the kernel}

Now we study the projection of the error of the approximate solution on the
kernels. We have, in the region where $\left\vert z-\varepsilon^{-1}%
p_{k}\right\vert \leq\frac{1}{2}\varepsilon^{-1},$
\begin{align*}
\nabla u_{k}\star\nabla u_{j}  &  =\partial_{x}u_{k}\partial_{x}u_{j}%
+\partial_{y}u_{k}\partial_{y}u_{j}\\
&  \sim\partial_{x}u_{k}\left(  -\frac{y_{j}ie^{i\theta_{j}}}{r_{j}^{2}%
}\right)  +\partial_{y}u_{k}\left(  \frac{x_{j}ie^{i\theta_{j}}}{r_{j}^{2}%
}\right)  .
\end{align*}
We have%
\[
\operatorname{Im}\left(  \partial_{y}u_{k}\overline{\left(  \partial_{x}%
u_{k}\right)  }\right)  =\frac{SS^{\prime}}{r_{k}}.
\]
It follows that
\begin{align*}
&  \operatorname{Re}\int_{\left\vert z-\varepsilon^{-1}p_{k}\right\vert \leq
r_{j}}e^{-i\theta_{j}}\left(  \nabla u_{k}\star\nabla u_{j}\right)  \left(
\overline{\partial_{x}u_{k}}\right)  dxdy\\
&  =-\operatorname{Re}\left(  \frac{\varepsilon}{p_{k}-p_{j}}\right)
\operatorname{Im}\int_{\left\vert z-p_{k}\right\vert \leq r_{j}}%
\operatorname{Im}\partial_{y}u_{k}\overline{\left(  \partial_{x}u_{k}\right)
}\\
&  =-\operatorname{Re}\left(  \frac{\varepsilon}{p_{k}-p_{j}}\right)
\int_{\left\vert z-\varepsilon^{-1}p_{k}\right\vert \leq r_{j}}\frac
{SS^{\prime}}{r_{k}}\\
&  =-\pi\operatorname{Re}\left(  \frac{\varepsilon}{p_{k}-p_{j}}\right)
+O\left(  r_{k}^{-2}\right)  .
\end{align*}
Similarly,
\begin{align*}
&  \operatorname{Re}\int_{\left\vert z-\varepsilon^{-1}p_{k}\right\vert \leq
r_{j}}e^{-i\theta_{j}}\left(  \nabla u_{k}\star\nabla u_{j}\right)
\overline{\left(  \partial_{y}u_{k}\right)  }\\
&  =\pi\operatorname{Im}\left(  \frac{\varepsilon}{p_{k}-p_{j}}\right)
+O\left(  r_{k}^{-2}\right)  .
\end{align*}

Next, if $l,j\neq k,$ we estimate that for $\left\vert z-\varepsilon^{-1}%
p_{k}\right\vert \leq C\varepsilon^{-1},$
\begin{align*}
\left(  \nabla u_{l}\star\nabla u_{j}\right)  \overline{\left(  \partial
_{x}u_{k}\right)  }  &  \sim e^{-i\theta_{k}}\left(  \frac{y_{l}}{r_{l}^{2}%
}e^{i\theta_{l}}\frac{y_{j}}{r_{j}^{2}}e^{i\theta_{j}}+\frac{x_{l}}{r_{l}^{2}%
}e^{i\theta_{l}}\frac{x_{j}}{r_{j}^{2}}e^{i\theta_{j}}\right)  \left(
-\frac{y_{k}S}{r_{k}^{2}}+\frac{x_{k}S^{\prime}}{r_{k}}\right) \\
&  =O\left(  \frac{1}{r_{l}r_{j}}\right)  .
\end{align*}
Finally,
\begin{align*}
\operatorname{Re}\int_{\mathbb{R}^{2}}i\varepsilon\partial_{y}u_{k}%
\overline{\left(  \partial_{y}u_{k}\right)  }  &  =O\left(  \varepsilon
^{2}\right)  ,\\
\operatorname{Re}\int_{\mathbb{R}^{2}}i\varepsilon\partial_{y}u_{k}%
\overline{\left(  \partial_{x}u_{k}\right)  }  &  =\pi\varepsilon+O\left(
\varepsilon^{2}\right)  .
\end{align*}
Combing these estimates, we find that the projected equation at the main order
is $\left(  \ref{trans}\right)  $ with $\mu=1.$

\subsection{The nonlinear scheme}

We search a traveling wave solution $U$ of GP equation:%
\[
-i\varepsilon\partial_{y}U=\Delta U+U\left(  1-\left\vert U\right\vert
^{2}\right)  .
\]
After a rescaling, we get the equation
\[
-i\partial_{y}U=\Delta U+\varepsilon^{-2}U\left(  1-\left\vert U\right\vert
^{2}\right)
\]
Then $\phi$ should satisfy
\[
\varepsilon i\partial_{y}\left(  u+\phi\right)  =\Delta\left(  u+\phi\right)
+\left(  u+\phi\right)  \left(  1-\left\vert u+\phi\right\vert ^{2}\right)  .
\]
We write this equation as
\begin{equation}
-\Delta\phi+\left(  1-\left\vert u\right\vert ^{2}\right)  \phi
-2u\operatorname{Re}\left(  \bar{u}\phi\right)  =E\left(  u\right)
+\varepsilon i\partial_{x}\phi+N\left(  \phi\right)  . \label{L}%
\end{equation}
Here $N\left(  \phi\right)  $ is a higher order perturbation term and equals
$2u\left\vert \phi\right\vert ^{2}+\bar{u}\phi^{2}+\phi\left\vert
\phi\right\vert ^{2}.$ Denote the left hand side by $G\phi.$ Instead of
analyze the operator $G$ directly, we will study its conjugate operator,
possibly with different forms in different regions of $R^{2}.$ This technique
has already appeared in Section 2. We write $u$ as $\tilde{\rho}%
e^{i\tilde{\theta}},$ where $\tilde{\rho},\tilde{\theta}$ are real valued.
Observe that
\[
\mathbb{L}\eta:=e^{-i\tilde{\theta}}G\left(  e^{i\tilde{\theta}}\eta\right)
=-e^{-i\tilde{\theta}}\Delta\left(  e^{i\tilde{\theta}}\eta\right)  +\left(
1-\tilde{\rho}^{2}\right)  \eta-2\tilde{\rho}\operatorname{Re}\left(
\tilde{\rho}\eta\right)  .
\]
If we write $\eta=\eta_{1}+i\eta_{2},$ with $\eta_{i}$ being real function,
then
\begin{align*}
\mathbb{L}\eta &  =-\Delta\eta-2e^{-i\tilde{\theta}}\nabla\left(
e^{i\tilde{\theta}}\right)  \nabla\eta-e^{-i\tilde{\theta}}\Delta\left(
e^{i\tilde{\theta}}\right)  \eta\\
&  +\left(  1-\tilde{\rho}^{2}\right)  \left(  \eta_{1}+i\eta_{2}\right)
-2\tilde{\rho}^{2}\eta_{1}.
\end{align*}
This tells us the in the region $\left\vert z\right\vert >C\varepsilon^{-1},$
the real part of $\mathbb{L}\eta$ is a well behaved operator like $-\Delta
\eta_{1}+2\eta_{1}.$ As for the imaginary part, we recall that by our
definition of $S,$
\[
1-\tilde{\rho}^{2}=0,\text{ for }\left\vert z\right\vert \geq C_{1}%
\varepsilon^{-1}.
\]
Hence the imaginary part behaves like the Laplacian operator $-\Delta\eta
_{2}.$ Writing $\phi=e^{i\tilde{\theta}}\eta,$ equation $\left(
\ref{L}\right)  $ becomes
\[
\mathbb{L}\eta=\left(  E\left(  u\right)  +\varepsilon i\partial_{x}%
\phi+N\left(  \phi\right)  \right)  e^{-i\tilde{\theta}}.
\]
Let us denote the right hand side by $M.$

\begin{lemma}
$M$ has the following symmetry%
\[
M\left(  \bar{z}\right)  =\overline{M\left(  z\right)  },M\left(  z^{\ast
}\right)  =M\left(  z\right)  .
\]

\end{lemma}

\begin{proof}
This follows from the symmetry of the approximate solution $u.$
\end{proof}

We also recall the following result from Lemma 4.2 in \cite{LFH}:

\begin{lemma}
\label{h}Let $h$ satisfy%
\[
\Delta h+f\left(  z\right)  =0\text{, }h\left(  \bar{z}\right)  =-h\left(
z\right)  ,\left\vert h\right\vert \leq C,
\]
where $f$ satisfies
\[
\left\vert f\left(  z\right)  \right\vert \leq\frac{C}{\left(  1+\left\vert
z\right\vert \right)  ^{2+\sigma}},\sigma\in\left(  0,1\right)  .
\]
Then
\[
\left\vert h\left(  z\right)  \right\vert \leq\frac{C}{\left(  1+\left\vert
z\right\vert \right)  ^{\sigma}}.
\]

\end{lemma}

Now we introduce the functional framework adapted to the mapping property of
the linearized operator $\mathbb{L}.$. Following \cite{LFH}, we fix $\gamma
\in\left(  0,1\right)  $ and $\sigma\in\left(  0,\frac{1}{2}\right)  .$ for
$\eta=\eta_{1}+\eta_{2}i,$ we define \
\begin{align*}
\left\Vert \eta\right\Vert _{\ast} &  :=\sum_{j}\left\Vert u\eta\right\Vert
_{C^{2,\gamma}\left(  r_{j}<2\right)  }+\sum_{j}\left\Vert u\eta\right\Vert
_{C^{2,\gamma}\left(  r_{j}<3\right)  }\\
&  +\sum_{j}\left\Vert r_{j}^{1+\sigma}\eta_{1}\right\Vert _{L^{\infty}\left(
r_{j}>2\right)  }+\sum_{j}\left\Vert r_{j}^{2+\sigma}\nabla\eta_{1}\right\Vert
_{L^{\infty}\left(  r_{j}>2\right)  }\\
&  +\sum_{j}\left\Vert r_{j}^{\sigma}\eta_{2}\right\Vert _{L^{\infty}\left(
r_{j}>2\right)  }+\sum_{j}\left\Vert r_{j}^{1+\sigma}\nabla\eta_{2}\right\Vert
_{L^{\infty}\left(  r_{j}>2\right)  }.
\end{align*}
For $h=h_{1}+ih_{2},$ we define
\begin{align*}
\left\Vert h\right\Vert _{\ast\ast} &  :=\sum_{j}\left\Vert h\right\Vert
_{C^{0,\gamma}\left(  r_{j}<2\right)  }\\
&  +\sum_{j}\left\Vert r_{j}^{1+\sigma}h_{1}\right\Vert _{L^{\infty}\left(
r_{j}<2\right)  }+\sum_{j}\left\Vert r_{j}^{2+\sigma}h_{2}\right\Vert
_{L^{\infty}\left(  r_{j}<2\right)  }.
\end{align*}

We have the following projected linear theory.

\begin{proposition}
\label{Linear}Let $\varepsilon$ be small. Suppose
\[
\left\{
\begin{array}
[c]{l}%
\mathbb{L}\eta=h,\\
\operatorname{Re}\left(  \int_{\left\vert z-\varepsilon^{-1}p_{k}\right\vert
\leq\varepsilon^{-1}}\bar{\eta}e^{-i\tilde{\theta}}\partial_{x}u\right)
=0,k=1,...,6.\\
\operatorname{Re}\left(  \int_{\left\vert z-\varepsilon^{-1}p_{k}\right\vert
\leq\varepsilon^{-1}}\bar{\eta}e^{-i\tilde{\theta}}\partial_{y}u\right)
=0,k=1,...,6\\
\eta e^{i\tilde{\theta}}\text{ has the symmetry as }M.
\end{array}
\right.
\]
Then $\left\Vert \eta\right\Vert _{\ast}\leq C\left\vert \ln\varepsilon
\right\vert \left\Vert h\right\Vert _{\ast\ast}.$
\end{proposition}

\begin{proof}
This can be proved by using the linear theory of the standard $\pm1$ vortex
described in Section \ref{Pri} and Lemma \ref{h}. It can also be proved along
the same ideas as that of Lemma 5.1 in \cite{LFH}, using blow up and
contradiction arguments, which is in the spirit similar as that of
\cite{Manuel}. Since this type of results are by now more or less standard, we
omit the details.
\end{proof}

Now we are ready to prove our main theorem in this paper. Since technically
the method is quite similar to that of \cite{LFH}, we only sketch the main steps.

Setting $\phi=e^{i\tilde{\theta}}\eta$, we write the nonlinear problem as
\begin{equation}
\mathbb{L}\eta=e^{-i\tilde{\theta}}\left(  E\left(  u\right)  +\varepsilon
i\partial_{x}\phi+N\left(  \phi\right)  \right)  .\label{nonlinear}%
\end{equation}
The error $E\left(  u\right)  $ can be estimated by $\left\Vert E\left(
u\right)  \right\Vert _{\ast\ast}\leq C\varepsilon^{1-\sigma}.$ By Proposition
\ref{Linear}, the equation $\left(  \ref{nonlinear}\right)  $ can be solved
modulo projection on the kernel $\partial_{x}u,\partial_{y}u,$ using
contradiction argument. More precisely, let $\eta_{k}$ be cutoff functions
supported in the region $\left\vert z-\varepsilon^{-1}p\right\vert \leq
c\varepsilon^{-1},$ for a fixed small constant $c$ less than the distances
between any two roots of the Adler-Moser polynomials $A_{n},B_{n}.$ We can
find $c_{k},d_{k},\eta$ such that
\[
\mathbb{L}\eta=e^{-i\tilde{\theta}}\left(  E\left(  u\right)  +\varepsilon
i\partial_{x}\phi+N\left(  \phi\right)  \right)  +\sum_{k}\left(
c_{k}e^{-i\tilde{\theta}}\partial_{x}u+d_{k}e^{-i\tilde{\theta}}\partial
_{y}u\right)  \eta_{k}.
\]
Moreover, $\left\Vert \eta\right\Vert _{\ast}\leq C\left\vert \ln
\varepsilon\right\vert \varepsilon^{1-\sigma}.$ Projecting both sides on
$\partial_{x}u,\partial_{y}u$ and using the estimate of $\eta,$ we find that
$c_{k},d_{k}$ equal zero, is equivariant to that $p_{k},q_{k}$ satisfy the
system
\[
\left\{
\begin{array}
[c]{c}%
\sum\limits_{j\neq\alpha}\frac{1}{p_{\alpha}-p_{j}}-\sum\limits_{j}\frac
{1}{p_{\alpha}-q_{j}}=1+O\left(  \varepsilon^{\alpha}\right)  ,\text{ for
}\alpha=1,...,N,\\
\sum\limits_{j\neq\alpha}\frac{1}{q_{\alpha}-q_{j}}-\sum\limits_{j}\frac
{1}{q_{\alpha}-p_{j}}=O\left(  \varepsilon^{\alpha}\right)  ,\text{ for
}\alpha=1,...,N,
\end{array}
\right.
\]
for some $\alpha>0.$ Now using the nondegeneracy of the roots of the
Adler-Moser polynomial and the Lipschitz dependence of the $O\left(
\varepsilon^{\alpha}\right)  $ term on $p_{k},$ we can solve this system using
contraction mapping principle and get a solution $p_{k},q_{k},$ close to roots
of the Adler-Moser polymonial $a,b.$

\bigskip

\bigskip
\end{document}